\documentclass[11pt]{article}

\include{epic.sty}
\usepackage{amssymb,amsmath,amsthm,hyperref,srcltx,accents}
\usepackage{amsfonts,stmaryrd,mathrsfs}
\usepackage{paralist}
\usepackage{tikz}
\usepackage{times}
\usepackage{array}
\usepackage{color}
\usepackage{multirow}
\usepackage{graphicx}
\usepackage{algorithm,algorithmicx}
\usepackage{algpseudocode}

\usepackage{graphics}
\usepackage{epsfig}
\usepackage{float}
\usepackage{caption}
\usepackage{subcaption}

\newlength{\originalbase}
\newcommand{\spacing}[1]{\setlength{\baselineskip}{#1\originalbase}}

\newif\ifnotesw\noteswtrue

\newtheorem{theorem}{Theorem}[section]

\newtheorem{lemma}[theorem]{Lemma}

\newtheorem{corollary}[theorem]{Corollary}

\begin{document}
\spacing{1.2}
\parskip=+3pt

\def\proofend{\hfill$\Box$\medskip}
\def\Proof{\noindent{\bf Proof. }}
\newcommand{\ProofOf}[1]{\noindent{\bf Proof of {#1}. }}
\newenvironment{proofof}[1]{\ProofOf{#1}}{\hfill $\Box$ \medskip}

\def\proj{\mathrm{proj}}
\def\diam{\mathrm{diam}}

\title{Two-Dimensional Pursuit-Evasion in a Compact Domain with Piecewise Analytic Boundary}
\author{Andrew Beveridge\footnote{Department of Mathematics, Statistics and Computer Science, Macalester College, Saint Paul, MN 55105. \texttt{abeverid@macalester.edu}}~ and Yiqing Cai\footnote{The Institute for Mathematics and Its Applications, University of Minnesota, Minneapolis, MN 55455, \texttt{yiqingcai@ima.umn.edu }}}

\date{}

\maketitle

\newtheorem{definition}{Definition}

\newcommand{\style}[1]{{\emph{#1}}} 

\newcommand{\ie}{{\em i.e.}}
\newcommand{\eg}{{\em e.g.}}
\newcommand{\cf}{{\em cf.}}

\newcommand{\real}{{\mathbb R}}
\newcommand{\comp}{{\mathbb C}}
\newcommand{\nats}{{\mathbb N}}
\newcommand{\rats}{{\mathbb Q}}
\newcommand{\zed}{{\mathbb Z}}
\newcommand{\euc}{{\mathbb E}}
\newcommand{\Domain}{{\mathcal D}} 

\newcommand{\cP}{{\mathcal{P}}}
\newcommand{\cQ}{{\mathcal{Q}}}

\def\cat{\mbox{\sc{cat}}(0)}
\def\catk{\mbox{\sc{cat}}(K)}
\def\totcur{\kappa_{\mathrm{total}}}
\def\c{\tilde{c}}
\def\u{\tilde{u}}
\def\v{\tilde{v}}
\def\x{\tilde{x}}
\def\y{\tilde{y}}
\def\z{\tilde{z}}
\def\e{\tilde{e}}
\def\p{\tilde{p}}

\def\dmin{d_{\min}}

\begin{abstract}
In a pursuit-evasion game, a team of pursuers attempt to capture an evader. The players alternate turns, move with equal speed, and have full information about the state of the game. We consider the most restictive capture condition: a pursuer must become colocated with the evader to win the game. We prove two general results about pursuit-evasion games in topological spaces. First, we show that one pursuer has a winning strategy in any $\cat$ space under this restrictive capture criterion. This complements a result of Alexander, Bishop and Ghrist, who provide a winning strategy for a game with positive capture radius. Second, we consider the game played in a compact domain in Euclidean two-space with piecewise analytic boundary and arbitrary Euler characteristic. We show that three pursuers always have a winning strategy by extending recent work of Bhadauria, Klein, Isler and Suri from polygonal environments to our more general setting.
\end{abstract}

%
%
%
%

\section{Introduction}

A pursuit-evasion game in a domain $\Domain$ is played by a team of pursuers $p_1, p_2, \ldots, p_k$ and an evader $e$. The pursuers win if some $p_i$  becomes colocated with the  evader after a finite number of turns, meaning that the distance $d(p_i,e) = 0$. When this occurs, we say that $p_i$ \style{captures} $e$. 
We consider the discrete time version of the game, which proceeds in turns.  Initially, the pursuers choose their positions $p_1^0, p_2^0, \ldots, p_k^0$, and then the evader chooses his initial position $e^0$. In turn $t \geq 1$, each pursuer $p_i$ moves from her current position $p_i^{t-1}$ to a point  $p_i^{t} \in B(p_i^{t-1}, 1) = \{ x \in \Domain \mid d(p_i^{t-1},x) \leq 1 \}.$ If the evader has been captured, then the game ends with the pursuers victorious. Otherwise, the evader moves from $e^{t-1}$ to a point $e^{t} \in B(e^{t-1}, 1)$.  The evader wins if he remains uncaptured forever. We consider the full-information (full-visibility) game in which each player knows the environment and the locations of all the other players. Furthermore, the pursuers may coordinate their movements.

Turn-based pursuit games in simply connected domains have been well-characterized: one pursuer is sufficient to capture the evader.
Winning pursuer strategies have been found for environments in $\real^n$ \cite{sgall, KR}, and in simply connected  polygons \cite{isler05tro}. Taking a  topological viewpoint and using the weaker capture criterion $d(p,e) < \epsilon$ for some constant $\epsilon >0$,  
Alexander, Bishop and Ghrist \cite{Alexander_pursuit} proved that a single pursuer can capture an evader in any compact $\cat$ by heading directly towards the evader at maximum speed.
We provide an alternate strategy for a compact $\cat$ space that achieves  $d(p,e)=0$ in a finite number of turns.
Our winning pursuer strategy is the topological version of \style{lion's strategy}, which has been used successfully in $\euc^n$ \cite{sgall} and in simple polygons \cite{isler05tro}. We defer the description of this strategy to Section \ref{sec:lion}.

 \begin{theorem}
\label{thm:lion}
A pursuer $p$ using lion's strategy in a compact $\cat$ space $\Domain$, captures the evader $e$ by achieving $d(p,e)=0$ after at most   $\diam(\Domain)^2$ turns.
\end{theorem}

Theorem \ref{thm:lion} implies that a single pursuer can become colocated with an evader in a simply connected, compact domain $\Domain \subset \euc^2$.
But it is easy to construct compact domains that are evader win:  removing one large open set from the middle of a simply connected domain tips the game in the evader's favor. Indeed, the evader can keep this large obstruction between himself and the pursuer, indefinitely. Such an open set is called an \style{obstacle} or \style{hole} in the environment. It is not hard to show that adding a second pursuer to this two-dimensional domain gives the game back to the pursuers. Adding multiple obstacles creates a distinct topology, and it is natural to wonder how many pursuers are needed to capture an evader in such an  environment. 
The analogous question has been resolved for pursuit-evasion games in certain two-dimensional environments.  Aigner and Fromme \cite{aigner+fromme} proved that three pursuers are sufficient for pursuit-evasion on a planar graph. 
Bhadauria, Klein, Isler and Suri \cite{bhadauria+klein+isler+suri} showed that the analogous result holds  in a two-dimensional polygonal environment with polygonal holes. We generalize this three-pursuer result to  a broader class of topological spaces.

Our pursuit game takes place in a compact and path-connected domain $\Domain \subset \euc^2$.
The set $\Domain$  contains a finite set of 
 disjoint open obstacles $\mathcal{O} = \{ O_1, O_2, \ldots, O_k \}$. 
The boundary of the domain is $\partial\Domain = \{  \partial O_0,  \partial O_1, \partial O_2, \ldots, \partial O_k \}$ where we define $\partial O_0$ to be the outer boundary of $\Domain$, for convenience.
We place two conditions on the boundary. First, 
$\partial \Domain$ can be decomposed into a finite number of analytic curves $\gamma(t) = (x(t), y(t))$ for $0 \leq t \leq 1$, where each of $x(t),y(t)$ can locally be expanded as convergent power series.   Second,  we require that
$\partial D$ is a 1-manifold: for any $x \in \partial D$, there exists an $\epsilon>0$, such that $B(x, \epsilon) \cap \partial D$ is homeomorphic to $\real^1$. In other words, we  forbid  self-intersections. For brevity, we say that a domain $\Domain$ satisfying these properties is  \style{piecewise analytic}. 
We list three consequences of these conditions. First,  the number of singular points on the boundary is finite. Second, the absolute value of the curvature at the nonsingular points of $\partial D$ is bounded above by some constant $\kappa_{\max} >0$.
Third, there is a minimum separation $\dmin > 0$ between boundary components: $d(O_i, O_j) > \dmin$ for all $0 \leq i < j \leq k$. 
During the game, the pursuers will guard a sequence of geodesics; crucially, we will see in Section \ref{sec:proof} that  each of these geodesics is also piecewise analytic.
This brings us to our main result.

\begin{theorem}
\label{thm:3pursuers}
Three pursuers can capture an evader in  a compact domain in $\euc^2$ with piecewise analytic boundary. The number of turns required to capture the evader for a domain with $k$ obstacles is
$O( 2k \cdot \diam(\Domain) + \diam(\Domain)^2)$.
\end{theorem}

At a high level,  our winning three-pursuer strategy  builds directly on those found in \cite{aigner+fromme, bhadauria+klein+isler+suri}, and we are indebted to those previous papers. However, our geometric and topological approach is entirely new. In particular, our arguments are grounded in a careful investigation of the  convexity, curvature and homotopy classes of geodesic curves in our domain.    
And of course, Theorem \ref{thm:3pursuers} significantly extends the class of known three-pursuer-win domains.

\subsection{Related Work}

Pursuit-evasion games have enjoyed a long  research history. In the 1930s, Rado posed the Lion and Man game in which a lion hunts the man in a circular arena. The players move with equal speeds, and the lion wins if it achieves colocation. At first blush, it seems that lion should be able to capture man, regardless of the man's evasive strategy. However, Besicovitch showed that when the game is played in continuous time, the man can follow a spiraling path so that lion can get arbitrarily close, but cannot achieve colocation \cite{littlewood}. However, when lion and man move in discrete time steps, our intuition prevails: lion does have a winning strategy \cite{sgall}.

The interdisciplinary literature on pursuit-evasion games spans a range of settings and variations. Pursuit games have been studied in many environments, including  graphs, in polygonal environments and in topological spaces. Researchers have considered motion constraints such as  speed differentials between the players, constraints  on acceleration, and energy budgets. As for sensing models, the players may have full information about the positions of other players, or they may have incomplete or imperfect information. Typically, the capture condition requires achieving colocation,  a proximity  threshold, or sensory visibility (such as a non-obstructed view of the evader).
For an overview of pursuit-evasion on graphs, see the monograph \cite{bonato+nowakowski}. The papers \cite{chi} and \cite{KR} provide a nice introduction to pursuit in the polygonal setting.


The past decade has witnessed  a renaissance of pursuit-evasion results  in multiple disciplines. 
Prominent research efforts come from the robotics community, where pursuit-evasion in polygonal environments is a productive setting for exploring autonomous agents. Pursuit-evasion has also thrived in the graph theory community, where it is known as the game of Cops and Robbers. More recently, researchers have started  exploring pursuit-evasion games in topological spaces. This is a natural evolution for the study of pursuit-evasion games.  Indeed,  determining the number of pursuers required to capture an evader in a given environment becomes a question about its topology since the various loops and holes of the environment provide escape routes for the evader. 

The classic paper of Aigner and Fromme \cite{aigner+fromme} initiated the study of multiple pursuers versus a single evader on a graph. In this turn-based game, agents can move to adjacent vertices, and the cops win if one of them becomes co-located with the robber.  This paper introduced the \style{cop number} of a graph, which is the minimum number of pursuers (cops) needed to catch the evader (robber). Aigner and Fromme proved that the cop number of any planar graph is at most 3. This bound is tight, as the dodecahedron graph requires three cops. At a high level, their winning pursuer strategy proceeds as follows. Two cops guard distinct $(u,v)$-paths where $u,v$ are vertices of the graph $G$. This restricts the pursuer movement to a subgraph of $G$. The third pursuer then guards another $(u,v)$-path, chosen so that (1) the robber's movement is further restricted, and (2) one of the other cops no longer needs to guard its path. This frees up that cop to continue the pursuit. This process repeats until the 
evader is caught.

More recently, an analogous result was proven by Bhadauria, Klein, Isler and Suri \cite{bhadauria+klein+isler+suri} for pursuit-evasion games in a two-dimensional polygonal environment with polygonal holes. 
In this turn-based game, an agent can move to any point within unit distance of its current location. Like Aigner and Fromme, they use colocation as their capture criterion. Bhadauria et al.~prove that three pursuers are  sufficient for pursuit-evasion in this setting, and that this bound is tight. The pursuer strategy is inspired by the Aigner and Fromme strategy for planar graphs: two pursuers guard paths that confine the evader while the third pursuer takes control of another path that further restricts the evader's movement. Of course, the details of the pursuit and the technical proofs are quite different from the graph case. Their proofs make heavy use of the polygonal nature of the environment, both to find the paths to guard and to guarantee that their pursuit finishes in finite time.

Just as the proofs of Bhadauria et al.~were inspired by Aigner and Fromme, our proof of Theorem \ref{thm:3pursuers} is inspired by those for the polygonal environment. Bhadauria et al.~actually give two different winning strategies for three pursuers. At a high level, these strategies progress in the same way, but the tactics for choosing paths and how to guard them are different. Herein, we adapt their \style{shortest path strategy} to our more general setting. The topological environment introduces a distinctive set of  challenges to overcome. 
In particular, we do not have a finite set of polygonal vertices to use as a backbone for our guarded  paths. Instead,  we rely on homotopy classes to differentiate between paths to guard.  Looking beyond the high-level structure of our pursuer strategy, the arguments (and their technical details) in this paper are wholly distinct from those found in \cite{bhadauria+klein+isler+suri}, and our result applies to a much broader class of environments.  

Finally, we note that our result follows in the footsteps of other recent explorations of pursuit-evasion games in  general geometric and topological domains. Pursuit-evasion games in such spaces have applications in robotics, where agents must navigate and coordinate in  high dimensional   configuration spaces. Alexander, Bishop and Ghrist helped to pioneer this subject, studying pursuit-evasion games with the capture condition  $d(p,e) < \epsilon$ for some constant $\epsilon > 0$ (rather that colocation). 
In \cite{Alexander_pursuit}, Alexander, Bishop and Ghrist prove that a single pursuer can capture an evader in any compact $\cat$ space: the \style{simple pursuit} strategy of heading directly towards the evader is a winning strategy. 
In \cite{Alexander_curvature}, these  authors explore the simple pursuit strategy in unbounded $\catk$ spaces with positive curvature $K >0$, developing connections between evader-win environments and the total curvature of the pursuer's trajectory.
Finally, in \cite{Alexander_unbounded},
they consider pursuit games in unbounded Euclidean domains using multiple pursuers. They provide conditions on the initial configuration of the players that guarantee capture, generalizing (and amending) results of Sgall \cite{sgall} and Kopparty and Ravishankar \cite{KR}.


\subsection{Preliminaries}

We introduce some  notation and review some concepts and results from algebraic topology \cite{hatcher}. We then prove three lemmas about convex paths in two-dimensional compact regions with piecewise analytic boundary.

 A topological space is a set $X$ along with a collection of subsets of $X$, called \style{open sets} that satisfy a sequence of axioms \cite{armstrong}.  A map $f: X \rightarrow Y$ between two spaces is \style{continuous}  when the inverse image of every open set in $Y$ is open in $X$.  
A \style{path} $\Pi: [0,1] \rightarrow \Domain$ is a continuous map  from interval $[0,1]$ to $\Domain$,
with initial point $\Pi(0)$ and terminal point $\Pi(1)$.   The \style{length} $l(\Pi)$ of this path is its arc length in Euclidean space $\euc^2$.  A path $\Pi$ is a \style{loop} when $\Pi(0)=\Pi(1)$. A \style{simple path} has no  self-intersections, meaning that $\Pi$ is injective. By abusing notation, we write $x \in \Pi$ when $x = \Pi(t)$ for some $t \in [0,1]$. For $x,y \in \Pi$, we use $\Pi(x,y)$ to denote the subpath connecting these points. The space $X$ is \style{path-connected} if there exists a path between any pair of points $x,y \in X$.

A \style{homotopy} of paths is a family of maps $f_t: [0,1] \rightarrow X,\ t\in [0,1]$, such that the associated map $F:  [0,1]\times  [0,1] \rightarrow X$ given by $F(s,t)=f_t(s)$ is continuous, and the endpoints $f_t(0)=x_0$ and $f_t(1)=x_1$ are independent of $t$. The paths $f_0$ and $f_1$ are called \style{homotopic}.
The relation of homotopy on paths with fixed endpoints is an equivalence relation and we use $[f]$ to denote the  homotopy class of the curve $f$ under this relation. 
The set $[f]$ of loops in $X$ at the basepoint $x_0$ forms a group under path composition, called the \style{fundamental group} of $X$ at the basepoint $x_0$.
The space $X$ is \style{simply connected} when it is path-connected and its fundamental group is trivial. For example, a subspace $X$ of $\euc^2$ is simply connected if and only if it has the same homotopy type as a 2-disc.

We now turn to some geometric properties of paths in $\euc^2.$ The distance $d(x,y)$ between points $x,y \in X$ is the length of a shortest $(x,y)$-path in $X$. When restricting ourselves to $R \subset X$, we use $d_R(x,y)$ to denote the distance between these points in the subdomain. We will frequently consider a subdomain $R$ enclosed by two simple $(u,v)$-paths $\Pi_1, \Pi_2$. We denote such a set as $R[\Pi_1, \Pi_2] \subset X$.

 For a $C^2$ path $\gamma : [0,T] \rightarrow \euc^2$,
 its curvature at $\gamma(t)$ is defined as $\kappa(t) = \pm\frac{||\gamma'(t)\times\gamma''(t)||}{||\gamma'(t)||^3}$, with the sign positive if the tangent turns counterclockwise, and negative if the tangent turns clockwise. 
The smoothness of a piecewise analytic curve $\gamma: [0,1] \rightarrow \euc^2$ ensures that  its absolute curvature is bounded at its nonsingular points.  If $\gamma$ is piecewise $C^2$ and continuous, with $t_0 < t_1 < \cdots < t_n$ as the preimages of the singular points, then its \style{total curvature}  is  
$$\totcur(\gamma) = \sum_{i = 1}^{n}\int_{t_{i-1}}^{t_{1}}\kappa(t)dt + \sum_{i = 1}^{n} \theta_i$$ 
where $\theta_i$ is the exterior angle at $\gamma(t_i)$, and $\theta_n=0$ when $\gamma(t_0) \neq \gamma(t_n)$.
This brings us to the celebrated Gauss-Bonnet Theorem which relates the total curvature of a closed curve with the Euler characteristic of its enclosed region. In our setting,  the Euler characteristic equals $1-k$, where  $k$ is the number of obstacles in the region $R$.

\begin{theorem}
\label{thm: gauss bonnet}
[Gauss-Bonnet Theorem, cf.~\cite{docarmo}] Given a compact region $R \subset \euc^2$ with boundary $\partial R$, we have 
$$
\totcur(\partial R)= 2\pi\chi(R),
$$ 
where $\chi(R)$ is the Euler characteristic of $R$. 
\end{theorem}



We  use the Gauss-Bonnet Theorem to understand the effect of obstacles on  shortest paths. In particular, we will consider pairs of paths $\Pi_1, \Pi_2$ with shared endpoints $u,v$. These paths will be piecewise analytic (so that they have a finite number of singular points). Our goal is to prove that if the shortest $(u,v)$-path is not unique, then each shortest path must touch an obstacle in the given region. We begin with a definition of convexity, which we define for the broader family of piecewise $C^2$ smooth curves; an example is shown in Figure \ref{fig: convexity}(a).

\begin{definition}
 \label{def: convex boundary}
Let $\Pi:[0,1]\rightarrow \euc^2$ be a piecewise $C^2$ smooth  curve in $\euc^2$.  Then $\Pi$ is \style{convex} when the following holds for  any point $x\in \Pi\backslash\{u,v\}$: \\
(a) If $\Pi$ is $C^2$ smooth at $x$, then  the curvature at $x$ is nonpositive;\\
(b) If $\Pi$ is not $C^2$ smooth at $x$, then the tangent line at $x$ turns clockwise by an angle $0 \leq \theta  \leq \pi$.  
\end{definition}

 \begin{figure}
\begin{center}
\begin{tabular}{ccc}

\begin{tikzpicture}


\draw[domain=100:340,smooth,variable=\x, dashed] plot ({\x}:.5);

\draw[domain=30:120,smooth,variable=\x, dashed] plot ({\x}:.4 + \x/90);

\draw[dashed] (-20:.5) -- (30:.75);


\draw plot [smooth] coordinates  {(100:.5)  (120:.75)  (150:1) (180:1.25)  (230: 1.5) (310: 1.25) (0: 1.25) (45: 1.75) (60: 2.25) (75:2.5) (90: 2.5) (120:1.767)};

\node at (90:1.1) {\small $\Pi$};

\draw[fill] (100:0.5) circle (2pt);
\draw[fill] (120:1.767) circle (2pt);

\node[left] at (120:1.767) {$u$};
\node[below right] at (100:0.5) {$v$};

\end{tikzpicture}

&
\qquad \qquad \qquad
&

\begin{tikzpicture}

\draw (90:2) to [bend right] (180:1);
\draw (330:1.5) to [bend left] (180:1);

\draw (90:2) to [bend left] (20:.9);
\draw (330:1.5) to [bend left] (20:.9);

\draw[dashed] (130:1.32) -- (260:.68);

\draw[fill] (180:1) circle (2pt);
\draw[fill] (130:1.32) circle (2pt);
\draw[fill] (260:.68) circle (2pt);

\draw[fill] (90:2) circle (2pt);
\draw[fill] (330:1.5) circle (2pt);

\node[left] at (180:1) {$x$};
\node[left] at (130:1.32) {$y_1$};
\node[below] at (260:.75) {$y_2$};

\node[above] at (90:2) {$u$};
\node[below] at (330:1.5) {$v$};

\node at (110:2) {$\Pi_1$};
\node at (20:1.3) {$\Pi_2$};
\node at (-.2, .3) {$\Lambda$};

\end{tikzpicture}

\\
(a) && (b)
\end{tabular}
 \end{center}
  \caption{(a) A piecewise convex $(u,v)$-path $\Pi$. (b) Shortcutting a non-convex $(u,v)$-path $\Pi_1$.}
\label{fig: convexity}
 \end{figure}
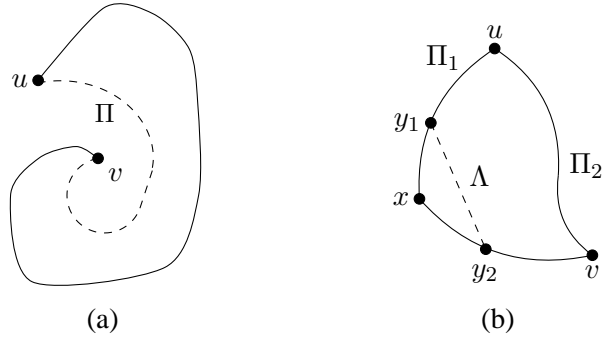





\begin{lemma}
 \label{lem: non touching shortest path in its homotopy is convex}
Let $\Pi_1$ and $\Pi_2$ be two piecewise analytic $(u,v)$-paths with $\Pi_1 \cap \Pi_2 = \{ u, v\}$.   If $\Pi_1$ is a shortest path in $R[\Pi_1,\Pi_2]$, and touches no obstacle inside $R[\Pi_1,\Pi_2]$, then $\Pi_1$ is convex in $R[\Pi_1,\Pi_2]$.
\end{lemma}

\begin{proof}
 We prove the lemma by contradiction. 
Suppose that there exists an $x\in\Pi \backslash \{ u,v \}$ where the convexity of $\Pi_1$ in $R[\Pi_1, \Pi_2]$ is violated. Either (a) $\Pi_1$ is  $C^2$ smooth at $x$, but the curvature at $x$ is positive, or (b) $\Pi_1$ is not $C^2$ smooth at $x$, and the tangent line turns counterclockwise by an angle $0 < \theta < \pi$ at $x$, creating a non-convex corner. 
Let $d_0$ denote the minimum distance between $\Pi_1$ and any obstacle $O \in R[\Pi_1,\Pi_2]$ with $\Pi_1 \cap \partial O = \emptyset$.

Suppose that the curvature at $x$ is positive, see Figure \ref{fig: convexity}(b). There must be a $C^2$ subpath 
 $\Pi_x$ between $y_1$ and $y_2$ of $\Pi_1$ around $x$ with positive curvature. Using the lower bound $d_0$ on the separation between $\Pi_1$ and any obstacles inside $R[\Pi_1,\Pi_2]$ and $\Pi_2$, we may take $y_1, y_2$ to be close enough so that the line segment $\Lambda$ connecting $y_1$ and $y_2$ lies inside $R[\Pi_1,\Pi_2]$ and does not encounter any obstacles.   
Replacing $\Pi_x$ with $\Lambda$ creates a path that is strictly shorter than $\Pi_1$,  contradicting the minimality of $\Pi_1$.

Next suppose  there is a non-convex corner at $x$. By an analogous argument to the previous case, we can create a short-cut $\Lambda$ around $x$ to make a shorter path than $\Pi_1$, a contradiction.
\end{proof}

\begin{lemma}
\label{lem: multiple convex lemma}
Let $\Pi_1, \Pi_2$ be two  $(u,v)$-paths with $\Pi_1 \cap \Pi_2 = \{ u, v \}$. Suppose that $\Pi_1$ is a convex and piecewise analytic $(u,v)$-path in  $R[\Pi_1,\Pi_2]$, and  let $\Pi_2$ be a convex and piecewise analytic  $(v,u)$-path in $R[\Pi_1,\Pi_2]$. Then $\Pi_1$ and $\Pi_2$ are both straight lines connecting $u, v$. 
\end{lemma}
\begin{proof}
Let $Q=Q[\Pi_1,\Pi_2]$, which is the closed region between $\Pi_1$ and $\Pi_2$, pretending there being no obstacles. If we consider $\partial Q$, the concatenation of $\Pi_1(u,v)$ and $\Pi_2(v,u)$, it is in fact a loop bounding a simply connected region. By the Gauss-Bonnet Theorem \ref{thm: gauss bonnet} , the total curvature along $\partial Q$  equals $2\pi \chi(Q)\geq2\pi$. We decompose the value $2\pi\chi(Q)$ as the sum of total curvature of $\Pi_1$ and $\Pi_2$ respectively, and the exterior angles at $u$ and $v$. Because of convexity, both $\Pi_1$ and $\Pi_2$ have total curvature no greater than $0$. As for the two angles at $u,v$, neither can exceed $\pi$. Therefore the total curvature of the loop does not exceed $2\pi$, and could only achieve $2\pi$ when $\totcur(\Pi_1) = \totcur(\Pi_2) = 0$. Therefore, $\Pi_1$ and $\Pi_2$ are both straight lines connecting $u$ and $v$. 
\end{proof}

\begin{lemma}
\label{lem: touching lemma}
Suppose $\Pi_1, \Pi_2$ are two shortest $(u,v)$-paths in region $R = R[\Pi_1,\Pi_2]$, with $\Pi_1\cap\Pi_2=\{u,v\}$. Then each of $\Pi_1, \Pi_2$ touches at least one obstacle in $R$.  
\end{lemma}
\begin{proof}
Suppose that the conclusion is false. Without loss of generality, $\Pi_1$ does not touch any obstacles in $R$. By Lemma \ref{lem: non touching shortest path in its homotopy is convex}, the $(u,v)$-path $\Pi_1$ is convex in $R$.
Let $Q$ be the simply connected region obtained by removing the obstacles in $R$. We have $l(\Pi_1)=l(\Pi_2)$, so they are both shortest $(u,v)$-paths in the simply connected environment $Q$. 
Therefore $\Pi_2$ is also convex by Lemma \ref{lem: non touching shortest path in its homotopy is convex}, if parameterized as a path from $v$ to $u$. By Lemma \ref{lem: multiple convex lemma}, $\Pi_1$ and $\Pi_2$ are both straight lines connecting $u,v$, which contradicts  $\Pi_1 \cap \Pi_2 = \{ u,v\}$. 
This  proves that when there is more than one shortest $(u,v)$-path, each of these paths must touch an obstacle.
\end{proof}

This concludes our topological and geometric preliminaries.

%
%
%

\setcounter{figure}{0}
\section{Lion's Strategy in a  $\cat$ space}
\label{sec:lion}

In this section, we describe a winning strategy for a single pursuer in a compact $\cat$ domain, and prove Theorem \ref{thm:lion}.  Our strategy generalizes \style{lion's strategy} for pursuit in  $\euc^2$, introduced by  Sgall \cite{sgall}. This strategy was adapted for pursuit in polygonal regions by Isler, Kannan and Khanna \cite{isler05tro}. Their adaptation relies heavily on the vertices of the polygon $P$ and gives a capture time of $n \cdot \diam(P)^2$, where $n$ is the number of vertices of $P$. We give a topological version of lion's strategy that succeeds in any compact CAT(0) domain $\Domain$ (including polygons) with capture time bounded by $\diam(\Domain)^2$.

Sgall's version of lion's strategy proceeds as follows. Fix a point $c$ as the center of our pursuit,
see Figure \ref{fig:sgall-lion}. The pursuer starts at $p=c$ and the evader starts at some point $e$. On her first move, the pursuer moves directly towards $e$ along the line $c e$. Considering a general round, the pursuer will be on the line segment between $c$ and $e$ prior to the evader move. After the evader moves to $e' \in B(e, 1)$,  the pursuer looks at the circle centered at $p$ with radius $\epsilon$. If $e$ is inside this circle, then the pursuer can capture the evader. Otherwise, this circle intersects the line segment $c e'$ at two points. The pursuer moves to the point $p'$ that is closer to $e$. 

\begin{lemma}[Sgall \cite{sgall}]
\label{lemma:sgall}
A pursuer using lion's strategy in $\euc^2$  re-establishes her location on the line segment between $c$ and the evader. Furthermore, if the evader moves from $e$ to $e'$ and the pursuer moves from $p$ to $p'$ then $d(c,p')^2 \geq d(c,p)^2 + 1$. \hfill $\Box$
\end{lemma}

\begin{figure}[t]

\begin{center}

\begin{tikzpicture}

\draw[fill] (0,0) circle (2pt);

\draw[fill] (10:4) circle (2pt);
\draw[fill] (10:6) circle (2pt);

\draw[dashed] (0,0) -- (10:7.5);
\draw[dashed] (0,0) -- (16:7.5);

\draw (10:4) circle (.75);

\draw[fill] (16:6.3) circle (2pt);
\draw[fill] (16:4.6) circle (2pt);

\draw[-latex] (10:6) -- (15.35: 6.28);
\draw[-latex] (10:4) -- (15.5: 4.55);

\node[above left] at (0,0) {$c$};

\node[below right] at (10:4) {$p$};
\node[above right] at (16:4.6) {$p'$};

\node[above] at (16:6.3) {$e'$};
\node[below] at (10:6)  {$e$};

\end{tikzpicture}

\caption{Lion's strategy in $\euc^2$. On each move, the pursuer moves on the line segment connecting the center $c$ to the evader, and increases her distance from $c$.}
\label{fig:sgall-lion}

\end{center}

\end{figure}
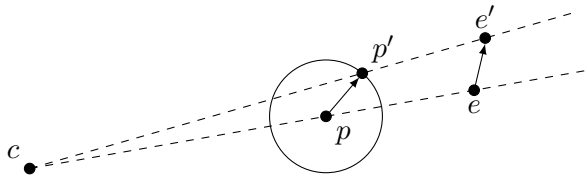

Before generalizing  lion's strategy, we introduce of the basics of a $\cat$ geometry; see \cite{bridson+haefliger} for a thorough treatment.
A complete metric space $(X,d)$ is a \style{geodesic space} when there is a unique path $\Pi(x,y)$ whose length is the metric distance $d(x,y)$. This path $\Pi(x,y)$ is called a geodesic (or shortest path). A \style{triangle} $\triangle xyz$ between three points $x,y,z \in X$ is the triple of  geodesics $\Pi(x,y), \Pi(y,z), \Pi(z,x)$.  To each $\triangle xyz \in X$, we associate a \style{comparison triangle} $\triangle \x \y \z \subset \euc^2$ whose side lengths in $\euc^2$ are the same as the lengths of the corresponding geodesics in $X$.  The complete geodesic metric space $(X,d)$ is $\cat$ when  no triangle in $X$ has a geodesic chord that is longer than the corresponding chord in the comparison triangle. In other words, pick any triangle $\triangle xyz$ and any points  $u \in \Pi(x,y)$ and $v \in \Pi(y,z)$. Let $\u \in \x\y$ and $\v \in \y\z$  be the corresponding points, chosen distancewise on the edges $\x\y$ and $\y\z$. If the space $X$ is $\cat$ then $d_X(u,v) \leq d_{\euc^2}(\u,\v)$. Colloquially, this is called the ``no fat triangles'' property, since it also implies that the sum of the angles of the triangle is not greater than $\pi$.

 Our $\cat$ lion's strategy generalizes the \style{extended lion's strategy} for polygons of  Isler et al.~\cite{isler05tro} by taking a more topological approach.
The pursuer starts at a fixed center point $c$ and her goal is to stay on the shortest path $\Pi(c,e)$ at all times.  In particular, assume that $p_t$ is on the shortest path $\Pi(c,e_t)$ and that the evader moves from $e_t$ to $e_{t+1}$. 
If $d(p_t, e_{t+1}) \leq 1$ then the purser responds by capturing the evader. Otherwise, the pursuer draws the unit circle $C$ centered at $p_t$ and moves to the point in $C \cap \Pi(c,e_{t+1})$ that is closest to $e_{t+1}$.
 
 \begin{lemma}[Lion's Strategy]
\label{lemma:lion}
A pursuer using lion's strategy in a $\cat$ space $(X,d)$  re-establishes her location on the line segment between $c$ and the evader. Furthermore, if the evader moves from $e$ to $e'$ and the pursuer moves from $p$ to $p'$ then $d(c,p')^2 \geq d(c,p)^2 + 1$. 

\end{lemma}

\begin{proof}
Suppose that $p \in \Pi(c,e)$ and then the evader moves to $e'$. Consider the $\cat$ triangle $\triangle c e e'$ and its comparison triangle
$\triangle \c \e \e'$ in $\euc^2$. Look at the corresponding $\euc^2$ pursuit-evasion game  with the pursuer at $\p \in \c \e$. By Lemma \ref{lemma:sgall}, the pursuer can move to a point $\p' \in \c \e'$ such that $d_{\euc^2}(\c,\p')^2 \geq d_{\euc^2}(\c,\p)^2 + 1$. Since there are no fat triangles in $X$, we have $d_{X}(p,p') \leq d_{\euc^2}(\p, \p')$ where $p' \in \Pi(c,e')$ is the point corresponding to $\p'$. Therefore, in our original game, the pursuer can move to the point $p' \in \Pi(c,e')$, which satisfies
$d_{X}(c,p')^2 \geq d_{X}(c,p)^2 + 1$.
\end{proof}

Finally, we prove Theorem \ref{thm:lion}: lion's strategy succeeds in a $\cat$ domain.

\begin{proofof}{Theorem \ref{thm:lion}}
Consider a pursuit-evasion game in the compact $\cat$ domain $\Domain$. Pick any $c \in \partial D$ as our center point. Using lion's strategy, the pursuer increases her distance from $c$ with every step by Lemma \ref{lemma:lion}, so she captures the evader after at most $\diam(\Domain)^2$ rounds.
\end{proofof}

%
%
%

\setcounter{figure}{0}
\section{Minimal Paths and Guarding}
\label{sec:guarding}

The key to our pursuit strategy is the ability of one pursuer to \style{guard} a shortest path, meaning that the evader cannot cross this path without being caught by a pursuer. When this shortest path splits the domain into two subdomains, the evader will be trapped in a smaller region. We refer to this region as the \style{evader territory}. 
In fact, we will be able to also guard a ``second shortest path'' when the shortest path is already guarded.  The definitions and lemmas in this section are adaptations of the minimal path formulations introduced in  \cite{bhadauria+klein+isler+suri} and further developed in \cite{ames}.
Recall that we use $d(x,y)$ to denote the length of a shortest $(x,y)$-path in $\Domain$. In addition, we will use $\mathring{X}$ and $\overline{X}$ to denote the interior and the closure of a set $X$, respectively.

\begin{definition}
Let $X \subset \Domain$ be a path-connected region. The simple path $\Pi \subset X$ is \style{minimal in $X$} when for any $y_1, y_2 \in \Pi$ and any $z \in X$, we have
$ d_{\Pi}(y_1,y_2)\leq d_{X}(y_1,z)+d_{X}(z,y_2)$.
\end{definition}

\begin{definition}
Let $Z \subset X$ and let $\Pi$ be a minimal $(u,v)$-path in $Z$ where $u,v \in \partial Z$. Then the \style{path projection with anchor $u$} is the function $\pi:Z \rightarrow \Pi$  defined as follows.  If $d_Z(u,z) < d_Z(u,v) = d_{\Pi}(u,v)$, then $\pi(z)$ is the  unique point $x \in \Pi$ with $d_{\Pi}(u,x) = d_Z(u,z)$. For the remaining $z \in Z \backslash \Pi$, we set $\pi(z) = v$.
\end{definition}

We make a few observations. If $X=\Domain$, then  a shortest $(x,y)$-path is always a minimal path in $\Domain$. In this case, we can define the path projection $\pi: \Domain \rightarrow \Pi$. When $X \subsetneq \Domain$, we might have $d_X(x_1, x_2) > d(x_1, x_2)$. In this case,  a shortest path in $X$ will be  minimal  in $X$, but it will not be minimal in $\Domain$. Next, we show that a path projection  is non-expansive, meaning that distances cannot increase. 

\begin{lemma}
\label{lemma:project}
Let  $\pi: Z \rightarrow \Pi$ be a path projection onto a minimal path in $Z$. Then  $d_{\Pi}(\pi(z_1),\pi(z_2)) \leq d_Z(z_1,z_2)$ for all $z_1, z_2 \in Z$. 
\end{lemma}

\begin{proof}
The proof is a straight-forward argument using the triangle inequality. We consider the case
$z_1, z_2 \in Z$ with $d(u,z_1) \leq d(u,z_2) \leq d(u,v)$.   We have
\begin{align*}
d_Z(z_2, z_1) 
&\geq
d_Z(z_2, u) - d_Z(z_1,u)  \, = \, d_{\Pi} (\pi(z_2),u) - d_{\Pi} (\pi(z_1),u) 
\, = \, d_{\Pi} (\pi(z_2), \pi(z_1)) .
\end{align*}
The other non-trivial cases are argued similarly.
\end{proof}

A single pursuer can turn a minimal path $\Pi$ into an impassable boundary:  once the pursuer has attained the position $p = \pi(e)$, the evader cannot cross $\Pi$ without being captured in response. The proof  of the following lemma is similar to the analogous result in \cite{ames}, but we include this brief argument for completeness. 

\begin{lemma}[Guarding Lemma]
\label{lemma:guarding}
Let $\pi: X \rightarrow \Pi$ be a path projection onto the minimal $(u,v)$-path $\Pi \subset X$.  Consider a pursuit-evasion game between pursuer $p$ and evader $e$ in the environment $X$. 
\begin{enumerate}
\item[(a)] After $O(\diam(X))$ turns, the pursuer can attain $p^t=\pi(e^{t-1})$.
\item[(b)] Thereafter, the pursuer can re-establish $p^{s+1} =\pi( e^{s})$ for all $s \geq t$.
\item[(c)]  If the evader moves so that a shortest path from $e^{s-1}$ to $e^s$ intersects $\Pi$, then the pursuer can capture the evader at time $s+1$.
\end{enumerate}

\end{lemma}

\begin{proof}
To achieve (a), the pursuer moves as follows. First, $p$ travels to $u$, reaching this point in $O(\diam(X))$ turns. Next, $p$ traverses along $\Pi$ until first achieving $d(u,p^i) \leq d(u,\pi(e^{i-1})) < d(u,p^i)+1$. If $p^i=\pi(e^{i-1})$ then we are done. Otherwise,  when the evader moves, we either have $d(u, p^i) -1 < d(u,\pi(e^i)) \leq d(u,p^i) + 1$ or $d(u,p^i) + 1 < d(u,\pi(e^i)) < d(u,p^i) +2$ by Lemma \ref{lemma:project}. In the former case, $p$ can move to $\pi(e^i)$ in response, achieving her goal. In the latter case, $p$ will increase her distance from $u$ by one unit, re-establishing $d(u,p^{i+1}) \leq d(u,\pi(e^{i})) < d(u,p^{i+1})+1$. This latter evader move can only be made  $O(\diam(X))$ times, after which the pursuer acheives $p=\pi(e)$. 

Next, suppose that $p^s=\pi(e^{s-1})$ and that $e^{s-1} \in X \backslash \Pi$. The pursuer can stay on the evader projection by induction since
$$
d_{\Pi}(p^t, \pi(e^t)) = d_{\Pi}(\pi(e^{t-1}), \pi(e^t)) \leq d(e^{t-1}, e^t) \leq 1, 
$$
so (b) holds. As for  (c),
suppose that a shortest path from $e^{s-1}$ to $e^s$ includes the point
$y \in \Pi$. Then
$$
d(p^t,e^t) \leq d_{\Pi} (\pi(e^{t-1}),y) + d(y,e^t) \leq d(e^{t-1}, y) + d(y,e^t) = d(e^{t-1},e^t)=1. 
$$
Therefore the pursuer can capture the evader on her next move.
\end{proof}

The Guarding Lemma is the cornerstone of our pursuer strategy.  When the pursuer moves as specified in the lemma, we say that she \style{guards} the path $\Pi$. In Section \ref{sec:proof}, our pursuers will repeatedly guard paths chosen to reduce the number of obstacles in the evader territory.  
Once the evader is trapped in a region that is obstacle-free, we have reached the endgame of the pursuit.

\begin{lemma}
\label{lemma:guardable pair}
 Suppose that the evader is located in a simply connected region $R$ whose boundary consists of subcurves of the original boundary $\partial \Domain$ and  two guarded paths $\Pi_1$ and $\Pi_2$. If the evader remains in $R$, then the third pursuer can capture him in finite time. If the evader tries to leave the region through $\Pi_1$ or $\Pi_2$, then he will be captured by the guarding pursuer.
 \end{lemma}

\begin{proof}
By Lemma \ref{lemma:guarding},  if the evader tries to leave this region, he will be caught by either $p_1$ or $p_2$. If the evader remains in this component, then Theorem \ref{thm:lion} guarantees that  pursuer $p_3$  captures the evader in a finite number of moves.
\end{proof}


The remainder of this section is devoted to  identifying guardable paths that touch obstacles. Guarding such a path will neutralize the threat posed by the obstacle. First, we consider the case when  $p_1$ guards the unique shortest $(u,v)$-path $\Pi$ that touches an obstacle $O$ in the evader region. The objective of $p_2$ is to guard another $(u,v)$-path $\Pi_2$ of a different homotopy type. This path can be guarded even when $\Pi_2$ is longer than $\Pi_1$, provided that any path shorter than $\Pi_2$  also intersects $\Pi_1$.

\begin{lemma}
\label{lem:next shortest path}
Suppose that the evader territory $R=[\Pi_1, \Delta]$ is bounded by the unique $(u,v)$-shortest path $\Pi_1$ and another boundary curve $\Delta$. Furthermore, suppose that $\Pi_1$ touches an obstacle $O$ and that $\Pi_1$ is guarded by $p_1$.
Then we can find a $(u,v)$-path $\Pi_2 \subset R$ with the following properties: (a) $O \subset R[\Pi_1, \Pi_2]$, so  that the homotopy type of $\Pi_2$  is different than that of $\Pi_1$; and  (b) $\Pi_2$ is guardable by $p_2$, provided that $\Pi_1$ remains guarded by $p_1$.
\end{lemma}

A naive attempt to find such a path is to pick some $x \in \Pi_1 \cap \bar{O}$ and find a shortest path that does not include the point $x$. However, $R \backslash \{ x \}$ is not a closed set, which would complicate our argument. Furthermore,  it could be that the next shortest path includes $x$ without using this point as a shortcut around the obstacle $O$, as shown in Figure \ref{fig:thru x}(b).\footnote{We note that this unusual circumstance is overlooked in \cite{bhadauria+klein+isler+suri}, where it can occur during their minimal path strategy. This case can be easily handled in a manner analogous to our approach, but based on the visibility graph of their  environment.} We handle both problems by removing a small and well-chosen open region $A$ near $x$, rather than removing the point $x$. The delicate choice of $A$ relies on the existence of both $\kappa_{\max}$ and $d_{\min}.$

\begin{figure}
\begin{center}

\begin{tikzpicture}[scale=2]

\begin{scope}[shift={(0.5,0)}]

\draw[fill] (0,2) circle (1pt);
\draw[fill] (0,.5) circle (1pt);

\node[left] at (0,2) {$u$};
\node[left] at (0,.5) {$v$};
\node[left] at (-.25, 1.25) {$x$};

\draw (0,.5) -- (-.15, 1.2) -- (-.15,1.3) -- (0,2);

\draw[fill=gray!30] (.1,1.25) ellipse (.25 and .35);

\node at (.15,1.25) {$O$};

\draw[fill] (-.15,1.25) circle (1pt);

\draw[dashed] (-.15, 1.25) circle (.11);

\draw plot [smooth] coordinates  {(0,.5)  (.5, .85) (.8,1) (.4,1 .85)  (0,2) };

\node at (.75,1.5) {$\Delta$};

\node at (-.2,.75) {$\Pi_1$};

\node at (.15,0) {(a)};

\end{scope}

\begin{scope}[shift={(2,0)}]

\draw[fill] (0,2) circle (1pt);
\draw[fill] (0,.5) circle (1pt);

\node[left] at (0,2) {$u$};
\node[left] at (0,.5) {$v$};
\node[above right] at (.5,1.5) {$x$};

\draw[fill=gray!30] (.5,1.5) -- (1, .75) -- (.4, .9) --cycle ;

\node at (.6, 1.025) {$O$};

\draw (0,.5) -- (.5, 1.5) -- (0,2);

\draw[very thick, dashed] (0,2) -- (.5, 1.5) -- (1, .75) -- (0,.5);

\node at (0,1) {$\Pi_1$};
\node at (.75,.5) {$\Pi_2$};

\node at (.25,0) {(b)};

\node at (1.2,1.5) {$\Delta$};

\draw[fill] (.5,1.5) circle (1pt);

\draw plot [smooth] coordinates  {(0,.5)  (1, .3) (1.2,1) (.8,1 .85)  (0,2) };

\end{scope}

\begin{scope}[shift={(4,0)}]

\draw[fill] (0,2) circle (1pt);
\draw[fill] (0,.5) circle (1pt);

\node[left] at (0,2) {$u$};
\node[left] at (0,.5) {$v$};
\node[above right] at (.5,1.5) {$x$};

\draw[fill=gray!30] (.5,1.5) -- (1, .75) -- (.4, .9) --cycle ;

\node at (.2, 1.3) {$A$};

\draw[-latex] (.21, 1.2) -- (.4, 1.1);

\draw (0,.5) -- (.5, 1.5) -- (0,2);


\draw[fill] (.5,1.5) circle (1pt);

\begin{scope}[shift={(.5, 1.5)}]

\draw[fill] (-99:.5) circle (1pt);
\draw[fill] (-117:.5) circle (1pt);

\draw (-99:.5) -- (-117:.5);

\node[left] at (-117:.5) {$y$};
\node[right] at (-99:.5) {$z$};

\end{scope}

\draw plot [smooth] coordinates  {(0,.5)  (1, .3) (1.2,1) (.8,1 .85)  (0,2) };

\node at (.25,0) {(c)};

\end{scope}

\begin{scope}[shift={(5.5,-0.5)}, scale=1.75]

\node[above right] at (.5,1.5) {$x$};

\draw[fill=gray!30] (.5,1.5) -- (1, .75) -- (.4, .9) --cycle ;

\node at (.3, 1.3) {$A$};

\draw (.125,.75) -- (.5, 1.5);

\begin{scope}[shift={(.5, 1.5)}]

\draw[fill] (-99:.5) circle (0.5pt);
\draw[fill] (-117:.5) circle (0.5pt);

\draw (-99:.5) -- (-117:.5);

\node[left] at (-117:.5) {$y$};
\node[right] at (-99:.5) {$z$};

\end{scope}

\draw[fill] (.5,1.5) circle (0.5pt);

\draw[dashed] plot [smooth] coordinates { (.15,.5) (.37, 1.15)  (.38,.7) (1, .6) };

\end{scope}

\node at (6.25,0) {(d)};

\end{tikzpicture}

\caption{Finding the second shortest path. (a) When $\Pi \cap \partial O$ contains a curve, we can remove a small open ball. (b) The shortest path $\Pi$ touches obstacle $O$ at $x$.  The second shortest path $\Pi$ goes around $O$, but includes the point $x$. (c) Finding $\Pi_2$ requires removing a small, open, triangular set $A$ between $\Pi_1$ and $O$, and then finding the shortest $(u,v)$-path in the closed set $R \backslash A$. (d) Any path that crosses the line segment $yz$ can be short-cut.}
\label{fig:thru x}
\end{center}
\end{figure}
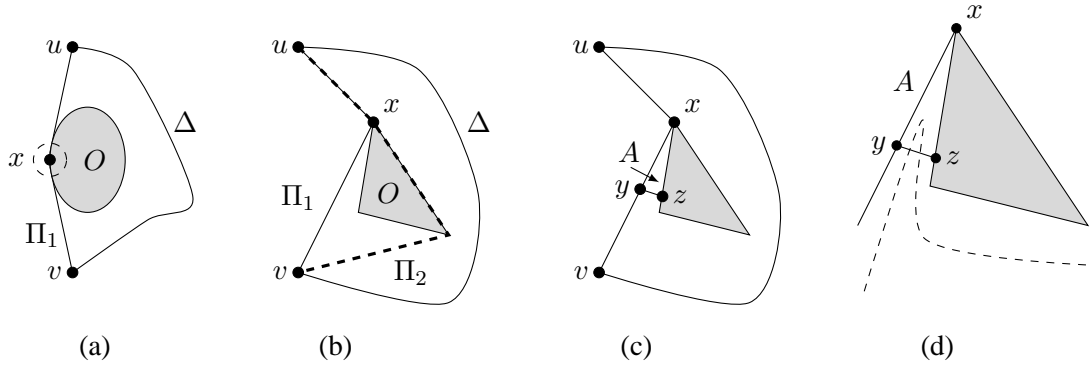

\begin{proof}
First, suppose that  $\Pi_1 \cap \partial O$ includes a continuous subcurve $C \subset \Pi_1$. Pick $x\in\mathring{C}$ and $\epsilon >0$ so that $B(x, \epsilon) \cap R \subset \overline{O}$. Let $R' = R \backslash \mathring{B}(x, \epsilon)$, which effectively absorbs the obstacle into the boundary, see Figure \ref{fig:thru x} (a). The region $R'$ is closed, so there is a well-defined shortest $(u,v)$-path $\Pi_2 \subset R'$.  The path $\Pi_2$ is guardable in $R'$, and therefore it is guardable by $p_2$ in $R$, provided that $p_1$  guards $\Pi_1$. Indeed, any shorter path in $R$ must go through the point $x$, so Lemma \ref{lemma:guarding} guarantees that an evader using such a path will be caught by $p_1$. Finally, we note that  $\Pi_1,\Pi_2$ have distinct homotopy types  because $O \subset R[\Pi_1, \Pi_2]$.

Next, we consider the case where $\Pi_1 \cap \bar{O}$ contains no continuous curves: we just focus on the first point $x \in \Pi_1 \cap \bar{O}$ encountered as we move from $u$ to $v$. Locally around $x$, the path $\Pi_1$ and the boundary $\partial O$ separate $R$ into two external regions (outside of $\Pi$ and inside $\partial O )$ and two internal regions, see Figure \ref{fig:thru x} (b). 
 The shortest path $\Pi_1$ does not self-intersect, so locally near $x$, this path consists of two line segments meeting at $x$, create an interior angle smaller than $2\pi$. Therefore, at least one of the two interior angles made by $\Pi_1$ and the obstacle tangent line(s) at $x$ is strictly less than $\pi$.   This local region is where we will remove our triangular open set.

 Without loss of generality, suppose that the subpath $\Pi_1(x,v)$ helps to bound this local region. 
Take points $y \in\Pi_1(x,v)$ and $z \in \partial O$ (traveling counterclockwise from $x$) such that $0 < d_{\Pi_1}(x,y) = d_{\partial O}(x,z) < \dmin/2$, and the angle $\angle yxz < \pi$.   Let $A\subset B(x, \dmin)$ be the closed region with endpoints $(x,y,z)$, where the third curve is the unique shortest $(y,z)$-path $\Gamma$, see Figure \ref{fig:thru x} (c). 
The bound on the absolute curvature $\kappa_{\max}$ allows us to choose our $y,z$ so that the region $A$ is essentially triangular. 
Since $\dmin$ is the minimum distance between obstacles, $A$ is obstacle-free, so $\Gamma$ is a straight line segment.

  We remove the relatively open set $A' = A \backslash \Gamma$ from our domain.  We then find the shortest $(u,v)$-path $\Pi_2$ in the closed set $R= R \backslash A'$. We claim that $p_2$ can guard $\Pi_2$ in $R$, provided that $p_1$ guards $\Pi_1$. As in the previous case, the shorter paths that go through $x$ are not available to the evader. Therefore, we must show that any path in $R$ that visits $A'$ is longer than $\Pi_2$. Such a path $\Lambda$ must enter and leave $A'$ through $\Gamma$, say at points $a,b$, see Figure \ref{fig:thru x} (d). However, the subpath  $\Lambda(a,b)$ can be replaced with the unique shortest path $\Gamma(a,b)$ without changing the homotopy type, a contradiction. 
Once again,  $\Pi_1,\Pi_2$ have distinct homotopy types  because $O \subset R[\Pi_1, \Pi_2]$.
\end{proof}

We refer to the paths $\Pi_1, \Pi_2$ from Lemma \ref{lem:next shortest path} as a \style{guardable pair}. Provided that the shortest $(u,v)$-path $\Pi_1$ is guarded, the ``second shortest $(u,v)$-path'' $\Pi_2$ can also be guarded. The following corollary is a  variation of the lemma.

\begin{corollary}
\label{cor:next shortest path}
Let $\Pi_1, \Pi_2$ be  $(u,v)$-paths that are guarded by $p_1, p_2$, respectively. Suppose that for $i=1,2$,  the path $\Pi_i$  touches an obstacle $O_i$, where $O_1 \neq O_2$.  
Then we can find a path $\Pi_3$ with the following properties: (a) the homotopy type of $\Pi_3$ is different than the homotopy types of $\Pi_1$ and $\Pi_2$, and in particular, $O_i \in R[\Pi_i, \Pi_3]$ for $i=1,2$; and (b) $\Pi_3$ is guardable by $p_3$, provided that $\Pi_1,\Pi_2$ remain guarded by $p_1, p_2$. 
\end{corollary}

\begin{proof}
The proof is similar to the proof of Lemma \ref{lem:next shortest path}. This time, we must remove an open set $A$ near $x \in \Pi_1 \cap O_1$ and an open set $B'$ near $y \in \Pi_2 \cap O_2$. We then find $\Pi_3$ in $R \backslash ( A' \cup B').$
\end{proof}

This concludes our search for guardable paths that touch obstacles. The next section lays out the three-pursuer strategy for capturing the evader in a two-dimensional domain.

%
%
%

\setcounter{figure}{0}
\section{Shortest Path Strategy: Proof of Theorem \ref{thm:3pursuers}}
\label{sec:proof}

In this section, we prove Theorem \ref{thm:3pursuers}:  three pursuers can capture an evader in a two-dimensional compact domain with piecewise analytic boundary. We adapt the  the shortest path strategy of Bhaudaria et al.~\cite{bhadauria+klein+isler+suri} to our more general topological setting. In particular, our  guardable path lemmas  from Section \ref{sec:guarding} supplant their use of polygon vertices to  find successive paths. Their algorithm guarantees success by reducing the number of polygon vertices in the evader territory. Instead, we keep track of the threat level of obstacles to argue that  the evader becomes trapped in a simply connected region.

Our pursuit proceeds in rounds. At the start of a round, at most two pursuers guard paths. The third pursuer moves to guard another path with the goal of eliminating obstacles from the evader territory. This third path will either be a shortest path, or it will create a guardable pair with the currently guarded path(s). Once this third path is guarded, the evader is trapped in a smaller region, which releases one of the other pursuers to continue the process. 
This continues until the evader is trapped in a simply connected region, where the free pursuer can capture the evader by Lemma \ref{lemma:guardable pair}.

We start by showing that the boundary of the evader territory is always piecewise analytic, after recalling two definitions. First, the endpoints of a line segment touching the boundary $\partial \Domain$ are called \style{switch points}.  Second, a point $x$ is an \style{accumulation point} (or limit point) of a set $S$ when any open set containing $x$  contains an infinite number of elements in $S$. We make use of the following result about the interaction of a geodesic with the boundary of the domain.

\begin{theorem}
[Albrecht and Berg \cite{albrecht+berg}]
 \label{thm: no accumulation point}
 If $M$ is a 2-dimensional analytic manifold with boundary embedded in $\euc^2$, and $\gamma$ is a geodesic in $M$, then the switch points on $\gamma$ have no accumulation points.
\end{theorem}

We  restrict ourselves to analytic boundary, instead of smooth ($C^2$, or even $C^{\infty}$) boundary, to avoid some potentially pathological behavior of geodesics. For example, Albrecht and Berg \cite{albrecht+berg} construct a geodesic in $C^{\infty}$ environment,  that achieves a Cantor set of positive measure as the accumulation of switch points. This unusual geometry hampers our ability to confine the evader in a well-defined  connected component.
Theorem \ref{thm: no accumulation point} ensures that our new evader territory will be bounded by piecewise analytic curves.  

\begin{lemma}
 \label{lem: shortest path is nice}
Let $\Domain$ be a compact domain with piecewise analytic boundary. If $\Pi$ is a shortest path in $\Domain$, then $\Pi$ is piecewise analytic. Furthermore, if $\Domain\backslash\Pi$ is disconnected, then it contains finitely many connected components, and the boundary of each connected component is piecewise analytic.
\end{lemma}

\begin{proof}
Let  $B \subset \partial(\Pi\backslash\partial\Domain)$ be the set of  \style{switch points}. We claim that $B$ is finite. Otherwise, there must be an accumulation point of $B$ since $l(\Pi)$ is finite,  contradicting Theorem \ref{thm: no accumulation point}. Now we can use the finite set $B$ as endpoints to partition $\Pi$ so that each subcurve is either in the boundary $\partial\Domain$, or in the interior $\mathring{\Domain}$.  Since any shortest path in the interior $\mathring{\Domain}$ must be a line segment, the path $\Pi$ is piecewise analytic.
For each connected component of $\Domain\backslash\Pi$, its boundary is a subset of $\Pi \cup \partial\Domain$, hence is piecewise analytic.  
\end{proof}

In order to prove Theorem \ref{thm:3pursuers}, we will show that our pursuit succeeds in finite time. To aid in this effort, we assign  a threat level to each of the $k$ obstacles in the original domain. These threat levels will reliably decrease during pursuit. An obstacle is in one of three states: dangerous, safe, or removed. A \style{removed} obstacle  lies outside the evader territory.
A \style{safe} obstacle lies in the evader territory and touches a currently guarded path. This obstacle is not a threat because the evader cannot circle around the object without being captured. The remaining  obstacles are \style{dangerous}. 
Finally, we say that the evader territory is dangerous if it contains at least one dangerous obstacle.

At the start of pursuit, all obstacles are dangerous.  So long as there are still dangerous obstacles, a round consists of taking control of a guardable path. This effort succeeds in a finite number of moves by Lemma \ref{lemma:guarding}.   We will show that after at most two rounds, either a dangerous obstacle transitions to safe/removed, or a safe obstacle transitions to removed.     This is our notion of progress: after at most $2k$ rounds, the evader territory is not dangerous. 
From here forward, we focus on the transition of the threat levels of obstacles.

In general, our evader territory will be bounded by part of the domain boundary $\partial \Domain$ and by at most two guarded paths $\Pi_1, \Pi_2$. At the end of a round, the evader territory will be updated, bounded in part by updated paths $\Pi_1',\Pi_2'$. If these guarded paths  intersect or share subpaths, then the evader is actually trapped in a smaller region by Lemma \ref{lemma:guarding}. When this is the case, we advance the endpoint(s) of our paths so that these are the only point(s) shared by our paths. This obviates the need to discuss degenerate cases.

The first round is an initialization round, so all obstacles might still be dangerous when this round completes. However, we will be able to neutralize at least one obstacle in the subsequent round. To kick off the first round, we pick points $u, v \in \partial \Domain$, chosen so that they divide the outer boundary into two curves $\Delta_1, \Delta_2$ of equal length, see Figure \ref{fig:round1}(a). Let $\Pi_1$ be a shortest $(u,v)$-path; if there are multiple shortest paths (in which case each touches an obstacle), then we pick one arbitrarily. 
Using Lemma \ref{lemma:guarding}, $p_1$ moves to guard $\Pi_1$. The round ends when $p_1$ has attained guarding position, trapping the evader in a subdomain that is bounded by $\Pi_1$ and one of $\Delta_1, \Delta_2$. The evader could be trapped in a smaller pocket region between $\Pi_1$ and a subcurve $\Delta_3$ of an obstacle $O \subset \Domain$, see Figure \ref{fig:round1}(b). In the latter case, the obstacle $O$ is marked as removed and we treat $\Delta_3$ as the 
new 
outer boundary. After updating the evader territory $R$, any obstacle $O \not\subset R$ is marked as removed. Any obstacle $O \subset R$ that touches $\Pi_1$ or $\Pi_2$ is marked as safe. 

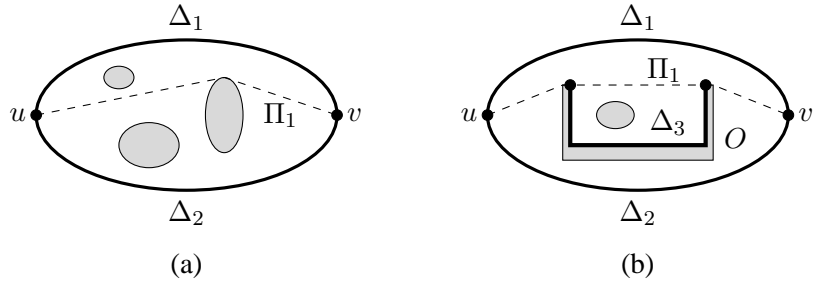
\begin{figure}
\begin{center}
\begin{tikzpicture}

\draw[very thick] (0,0) ellipse (2 and 1);

\draw[fill=gray!30]  (.5,0) ellipse (.25 and .5);

\draw[fill=gray!30]  (-.9,.5) ellipse (.2 and .15);

\draw[fill=gray!30]  (-.5,-.4) ellipse (.4 and .3);

\draw[dashed] (-2,0) -- (.5, .5)  -- (2,0);

\node[left] at (-2,0) {$u$};
\node[right] at (2,0) {$v$};

\draw[fill] (-2,0) circle (2pt);
\draw[fill] (2,0) circle (2pt);

\node at (90:1.3) {$\Delta_1$};

\node at (270:1.3) {$\Delta_2$};

\node at (1.25,0) {$\Pi_1$};

\node at (0, -2) {(a)};

\begin{scope}[shift={(6,0)}]

\draw[very thick] (0,0) ellipse (2 and 1);

\draw[fill=gray!30]  (1,.4) -- (1, -.6) -- (-1, -.6) -- (-1, .4) -- (-.9, .4) -- (-.9, -.4) -- (.9,-.4) -- (.9, .4) -- cycle;

\draw[fill=gray!30]  (-.3,0) ellipse (.25 and .18);

\draw[dashed] (-2,0) -- (-1, .4) -- (1, .4) -- (2,0);

\node[left] at (-2,0) {$u$};
\node[right] at (2,0) {$v$};

\draw[fill] (-2,0) circle (2pt);
\draw[fill] (2,0) circle (2pt);

\draw[fill] (-.9,.4) circle (2pt);
\draw[fill] (.9,.4) circle (2pt);

\draw[ultra thick] (-.9,.4) -- (-.9, -.4) -- (.9,-.4) -- (.9, .4);

\node at (90:1.3) {$\Delta_1$};

\node at (270:1.3) {$\Delta_2$};

\node at (.4,-.1) {$\Delta_3$};

\node at (.35,.6) {$\Pi_1$};
\node at (1.3,-.3) {$O$};

\node at (0, -2) {(b)};

\end{scope}

\end{tikzpicture}

\end{center}

\caption{The shortest $(u,v)$-path $\Pi_1$ guarded in round one. (a) The path $\Pi_1$ partitions the outer boundary to subcurves $\Delta_1, \Delta_2$ of equal length. (b) The evader may be trapped in a pocket between the path $\Pi_1$ and the boundary subcurve $\Delta_3$ of obstacle $O$. }
\label{fig:round1}

\end{figure}

For the remainder of the game, the boundary of the evader territory is one of the following types.
\begin{itemize}
\item Type $0$ region: A region containing no dangerous obstacles.
\item Type $1$ region: A dangerous three-sided region bounded by a $(u,v)$-shortest path $\Pi_1$, a $(u,w)$-shortest path $\Pi_2$ and a $(v,w)$-path $\Delta \subset \partial D$. No obstacle touches both $\Pi_1, \Pi_2$.
\item Type $1'$ region: A dangerous two-sided region bounded by a $(u,v)$-shortest path $\Pi_1$ and a $(u,v)$-path $\Delta \subset \partial \Domain$.  We  treat this as  a special case of the previous type, where $\Pi_2$ consists of the single point $w=u$. This point is on $\Pi_1$, so it is guarded by $p_1$.
\item Type 2 region: A dangerous two-sided region bounded by  $(u,v)$-paths $\Pi_1, \Pi_2$, each of which touches an obstacle in the evader territory.   The path $\Pi_1$ is a shortest $(u,v)$-path in this region. The path $\Pi_2$ might also be a shortest $(u,v)$-path, or it could be a ``second shortest path,'' meaning that it is a shortest $(u,v)$-path among the set of $(u,v)$-paths that are not homotopic to $\Pi_1$. No obstacle touches both $\Pi_1, \Pi_2$.
\item Type 3 region: a dangerous 4-sided region bounded by a $(u,v)$-shortest path $\Pi_1$, a $(w,x)$-shortest path $\Pi_2$, a $(v,w)$-path $\Delta_1$ from the boundary and a $(u,x)$-path $\Delta_2$ from the boundary. These vertices are arranged so that they are ordered clockwise as $u,v,w,x$.  No obstacle touches both $\Pi_1, \Pi_2$.
\end{itemize}
For example, after the initialization round, the evader territory is a type $1'$ region, bounded by a guarded path and part of the boundary $\partial \Domain.$ Finally, we emphasize that Lemma
\ref{lem: shortest path is nice} ensures that the boundary of the evader region is always piecewise analytic, since it consists of sub-curves of the piecewise analytic boundary along with one or more shortest paths.

We now describe the different types of rounds. In regions of  type 1, $1'$ and 2, we will always transition at least one obstacle. At the end of such a round, the evader could now be trapped in a region of any type. Type 3 rounds are slightly different. Our primary goal is to trap the evader in  a type 1 region, where we will surely make progress in the subsequent next round. However,  it is possible to transition an object via a type 3 move (just as in the initialization round). In this case, we make immediate progress, and the evader could then be trapped in a region of any type.


First we consider type 1 regions. This also handles  type $1'$ regions as a special case. 
We use the following lemma to identify a point $x \in \Delta$ and a shortest $(u,x)$-path to guard during this round. 

\begin{lemma}
\label{lem: multiple shortest paths}
 Let shortest paths $\Pi_0(u,v),\Pi_1(u,w)$ and boundary path $\Delta(v,w)$ bound a type $1$ environment $R$. If $R$ contains obstacles, then there exists a point $x\in\Delta$ such that there are multiple shortest $(u,x)$-paths in $R$, each of which touches at least one obstacle.
\end{lemma}

\begin{proof}
Parameterize  the boundary path as $\Delta: [0,1]\rightarrow R$.  We prove the lemma by contradiction. 

Suppose that for every $t\in[0,1]$, the shortest $(u,\Delta(t))$-path $\Pi_t$ is unique. Denote its length by $l(\Pi_t) = d(u,\Delta(t))$.  Define the function $n(t)$ to be the number of obstacles in  the region $R_t$ bounded by $\Pi_0,\Pi_t$ and $\Delta$. The function $n(t)$ is well-defined by the uniqueness of each $\Pi_t$. Furthermore,  $n(0)=0$, and $n(1)>0$, so there must be a jump discontinuity somewhere in $[0,1]$. Let $s=\inf\{t\in[0,1]: n(t)>0\}$. 
 
 Case 1: $n(s)>0$ where $0 < s \leq 1$. (Recall that  $n(0)=0$, so $s>0$.)  Let $\Gamma$ be a shortest $(u,\Delta(s))$-path that is of the same homotopy class as $\Pi_0$. 
 The choice of $s$ and the uniqueness of $\Pi_s$ guarantee that $l(\Gamma) > l(\Pi_s)$. Also no obstacles are contained in the region bounded by $\Pi_0,\Gamma$ and $\Delta(0,s)$. 
By the continuity of $l(\Gamma_t) = d(u,\Delta(t))$ with respect to $t$, we have
 $l(\Gamma) \leq l(\Pi_t)+l(\Delta(t,s))$, so
 $$
 l(\Gamma) \leq \lim_{t \rightarrow s^-}  \big( l(\Pi_t)+l(\Delta(t,s)) \big) = l(\Pi_s) + 0 = l(\Pi_s). 
 $$
This contradicts $l(\Gamma) > l(\Pi_s)$, so   $\Pi_s$ is not the unique shortest $(u, \Delta(s))$-path.
 
 Case 2: $n(s)=0$, where  $0 \leq s < 1$. Let $\{ s_i \}$ be an infinite sequence $s_i \rightarrow s^+$, such that $n(s_i)>0$ for all $i$. 
There are finite number of obstacles, so by taking a subsequence if necessary, we can assume that the shortest paths $\{ \Pi_{s_i} \}$ are of the same homotopy class. 
Let $\Gamma$ be the shortest $(u, \Delta(s))$-path of this homotopy class. We have
$ l(\Gamma)  \leq   l(\Pi_{s_i}) +  l(\Delta(s,s_i))$ for all $i$, and therefore
 $$
 l(\Gamma)  \leq   \lim_{i \rightarrow \infty} \big( l(\Pi_{s_i})+  l(\Delta(s,s_i)) \big)
 = l(\Pi_s),
 $$
 where the limit holds by the continuity of distances in the region. However, this contradicts the uniqueness of $\Pi_s$ which would require
 $l(\Pi_s) < l(\Gamma)$.
 
 Finally, we can conclude that are  multiple shortest $(u,x)$-paths. By Lemma \ref{lem: touching lemma},  each of these paths touches at least one obstacle.
\end{proof}

Having found the next path to guard, we now prove that we transition an object during a type 1 move.

\begin{lemma}
\label{lemma:type1region}
Suppose that the evader is trapped in a type 1 (or type $1'$) region. Then the third pursuer can guard a path that transitions an obstacle state.
\end{lemma}

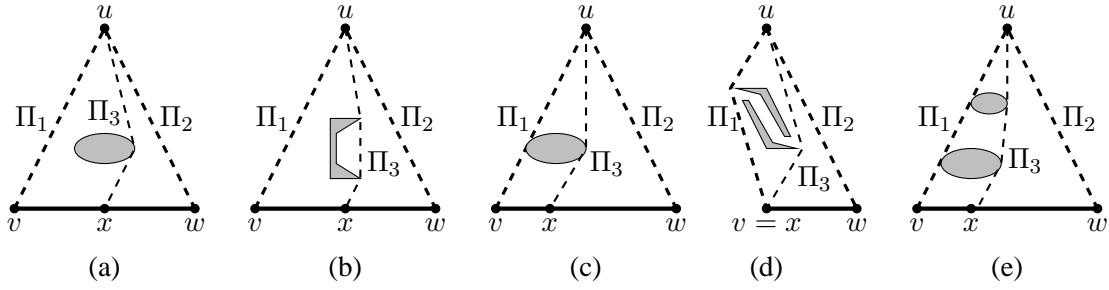
\begin{figure}
\begin{center}

\begin{tikzpicture}[scale=.8]

\draw[dashed, very thick] (-1.5,0) -- (0,3) -- (1.5,0);
\draw[ultra thick] (-1.5,0) -- (1.5,0);

\node[above] at (0,3) {$u$};
\node[below] at (-1.5,0) {$v$};
\node[below] at (1.5,0) {$w$};
\node[below] at (0,0) {$x$};

\draw[fill] (0,3) circle (2pt);
\draw[fill] (-1.5,0)  circle (2pt);
\draw[fill] (1.5,0) circle (2pt);
\draw[fill] (0,0) circle (2pt);

\draw[fill=gray!50] (0,1) ellipse (.5 and .25);


\draw[dashed, thick] (0,3) -- (.5,1) -- (0,0);

\node at (-1.2,1.5) {$\Pi_1$};
\node at (0,1.6) {$\Pi_3$};
\node at (1.2,1.5) {$\Pi_2$};

\node at (0, -1) {(a)};

\begin{scope}[shift={(4,0)}]

\draw[dashed, very thick] (-1.5,0) -- (0,3) -- (1.5,0);
\draw[ultra thick] (-1.5,0) -- (1.5,0);

\node[above] at (0,3) {$u$};
\node[below] at (-1.5,0) {$v$};
\node[below] at (1.5,0) {$w$};
\node[below] at (0,0) {$x$};

\draw[fill] (0,3) circle (2pt);
\draw[fill] (-1.5,0)  circle (2pt);
\draw[fill] (1.5,0) circle (2pt);
\draw[fill] (0,0) circle (2pt);

\draw[fill=gray!50] (.25, 1.5) -- (-.15, 1.25) -- (-.15, .75)-- (.25, .5) -- (-.25, .5) -- (-.25, 1.5) -- cycle;


\draw[dashed, thick] (0,3)  -- (.25,1.5) -- (.25,.5) -- (0,0);

\node at (-1.2,1.5) {$\Pi_1$};
\node at (0.65,.75) {$\Pi_3$};
\node at (1.2,1.5) {$\Pi_2$};


\node at (0, -1) {(b)};

\end{scope}

\begin{scope}[shift={(8,0)}]

\draw[dashed, very thick] (-1.5,0) -- (0,3) -- (1.5,0);
\draw[ultra thick] (-1.5,0) -- (1.5,0);

\node[above] at (0,3) {$u$};
\node[below] at (-1.5,0) {$v$};
\node[below] at (1.5,0) {$w$};
\node[below] at (-.6,0) {$x$};

\draw[fill] (0,3) circle (2pt);
\draw[fill] (-1.5,0)  circle (2pt);
\draw[fill] (1.5,0) circle (2pt);
\draw[fill] (-.6,0) circle (2pt);

\draw[fill=gray!50] (-.5,1) ellipse (.5 and .25);


\draw[dashed, thick] (0,3)  -- (0,1) -- (-.6,0);

\node at (-1.2,1.5) {$\Pi_1$};
\node at (0.35,.75) {$\Pi_3$};
\node at (1.2,1.5) {$\Pi_2$};


\node at (0, -1) {(c)};

\end{scope}

\begin{scope}[shift={(11,0)}]

\draw[dashed, very thick]  (0,3) -- (1.5,0);
\draw[ultra thick] (0,0) -- (1.5,0);

\node[above] at (0,3) {$u$};

\node[below] at (1.5,0) {$w$};
\node[below] at (0,0) {$v=x$};

\draw[fill] (0,3) circle (2pt);
\draw[fill] (1.5,0) circle (2pt);
\draw[fill] (0,0) circle (2pt);

\draw[fill=gray!50] (-.5,2) -- (0,2) -- (.4, 1.2) -- (.3, 1.2) -- (-.1,1.9) -- cycle;
\draw[fill=gray!50] (.5,1) -- (0,1) -- (-.4, 1.8) -- (-.3, 1.8) -- (.1,1.1) -- cycle;


\draw[dashed,very thick] (0,3) -- (-.6,2) -- (0,0);
\draw[dashed, thick] (0,3) -- (.6,1) -- (0,0);

\node at (-.7,1.5) {$\Pi_1$};
\node at (.8,.5) {$\Pi_3$};
\node at (1.2,1.5) {$\Pi_2$};

\node at (0, -1) {(d)};

\end{scope}

\begin{scope}[shift={(15,0)}]

\draw[dashed, very thick] (-1.5,0) -- (0,3) -- (1.5,0);
\draw[ultra thick] (-1.5,0) -- (1.5,0);

\node[above] at (0,3) {$u$};
\node[below] at (-1.5,0) {$v$};
\node[below] at (1.5,0) {$w$};
\node[below] at (-.6,0) {$x$};

\draw[fill] (0,3) circle (2pt);
\draw[fill] (-1.5,0)  circle (2pt);
\draw[fill] (1.5,0) circle (2pt);
\draw[fill] (-.6,0) circle (2pt);

\draw[fill=gray!50] (-.3,1.75) ellipse (.3 and .175);

\draw[fill=gray!50] (-.6,.75) ellipse (.5 and .25);

\draw[dashed, thick] (0,3)  -- (0,1.75) -- (-.1,.75) -- (-.5, 0);

\node at (-1.2,1.5) {$\Pi_1$};
\node at (.3,.8) {$\Pi_3$};
\node at (1.2,1.5) {$\Pi_2$};

\node at (0, -1) {(e)};

\end{scope}

\end{tikzpicture}

\caption{Representative examples of a type 1 move, where we transition to (a) a type 1 region, (b)  a type 1 or type $1'$ region, (c)  a type 1 or a type 3 region, (d)  a type 1 or type 2 region, (e) a type 1 or type 3 region. 
 }

\label{fig:type1move}

\end{center}
\end{figure}

\begin{proof}
By Lemma \ref{lem: multiple shortest paths}, there is some point $x \in \Delta$ with multiple shortest $(u,x)$-paths, each of which touches an obstacle.  Let $\Pi_3$ be one of these shortest $(u,x)$-paths. If $x =v$  then we take a path $\Pi_3 \neq \Pi_1$. Similarly if $x=w$ we choose $\Pi_3 \neq \Pi_2$.  When $x \notin \{ v,w \}$,   we can choose $\Pi_3$ arbitrarily from the collection of $(u,x)$-shortest paths.  Pursuer $p_3$ moves to guard $\Pi_3$, which traps the evader in either $R[\Pi_1, \Pi_3]$ or   $R[\Pi_3, \Pi_2]$. Any obstacles in the other region are marked as removed.

Without loss of generality, let $O \subset R[\Pi_1, \Pi_3]$ be an obstacle touched by $\Pi_3$, see Figure \ref{fig:type1move}(a).  
Suppose that prior to $p_3$ guarding $\Pi_3$, the object $O$ was dangerous.  If $e \in R[\Pi_3, \Pi_2]$ then $O$ transitions to removed.  If $e \in R[\Pi_1, \Pi_3]$ then $O$ transitions to safe. 
However, we may be in a more advantageous position, shown in Figure \ref{fig:type1move}(b):  the evader could be trapped in a pocket between obstacle $O$ and path $\Pi_3$. In this case, the new evader territory is type $1'$ and the obstacle $O$ is marked as removed, since it is now part of the outer boundary of the evader territory. 

Next,  suppose that $O$ was already safe, touched by $\Pi_1$. If $e \in R[\Pi_2, \Pi_3]$ after $p_3$ guards $\Pi_3$, then $O$ transitions from safe to removed. 
If the evader is trapped in a pocket region between $O$ and $\Pi_3$, we proceed as in the previous case. 
Otherwise, we have $e \in R[\Pi_1, \Pi_3]$ and the obstacle $O$
 separates $R[\Pi_1, \Pi_3]$ into disjoint regions, as shown in Figure \ref{fig:type1move}(c). The evader is trapped in one of these two subregions because both $\Pi_1, \Pi_3$ are guarded. Let $\Delta'$ be the subcurve of $\partial O$ that bounds the effective evader territory.  We update the evader territory appropriately,  bounded by $\Delta'$ and subpaths of $\Pi_1, \Pi_3$, and perhaps part of $\Delta$. The result is a region of type 1 or 3. The obstacle $O$ is marked as removed: it is now part of the boundary. This reduces the number of safe obstacles.

When $\Pi_3$ touches multiple obstacles,  each of them transitions to a lower threat level.  Figures \ref{fig:type1move}(d) and (e) show that we can also end up in a type 2 or 3 region, depending on the configuration of these obstacles and the location of the evader at the end of the round.
\end{proof}


Next, we consider a type 2 region. Such a region is bounded by   $(u,v)$-paths $\Pi_1, \Pi_2$ that form a guardable pair, where $\Pi_1,\Pi_2$ touch  safe obstacles $O_1, O_2$, respectively.  Without loss of generality,  $\Pi_1$ is a shortest $(u,v)$-path in the region, and $\Pi_2$ is either another shortest path, or a ``second shortest path'' as found in Lemma \ref{lem:next shortest path}.  (A type 1 move can lead to the first case. A type 2 move can lead to the second case, as we are about to see.)

\begin{lemma}
\label{lemma:type2region}
Suppose that the evader is trapped in a type 2 region. Then the third pursuer can guard a path that transitions an obstacle state.
\end{lemma}

\begin{proof}
Use Corollary \ref{cor:next shortest path} to find a guardable $(u,v)$-path $\Pi_3$ in $R[\Pi_1, \Pi_2]$ whose homotopy type is distinct from that of both $\Pi_1, \Pi_2$.
Pursuer $p_3$ establishes a guarding position on $\Pi_3$. The evader is now trapped in either $R[\Pi_1, \Pi_3]$ or $R[\Pi_3, \Pi_2]$, so one of $O_1,O_2$ transitions from safe to removed. Furthermore, $\Pi_3$ must touch at least one obstacle in each of $R[\Pi_1, \Pi_3]$ or $R[\Pi_3, \Pi_2]$. Otherwise, $\Pi_3$ would be shorter than one of $\Pi_1, \Pi_2$, which contradicts the minimality of that path in $R[\Pi_1, \Pi_2]$. 
Depending on the configuration of the obstacles, we may be able to restrict the evader territory further. After doing so, the evader territory may be of any possible type, as shown in Figure \ref{fig:type2-type3} (a).
\end{proof}

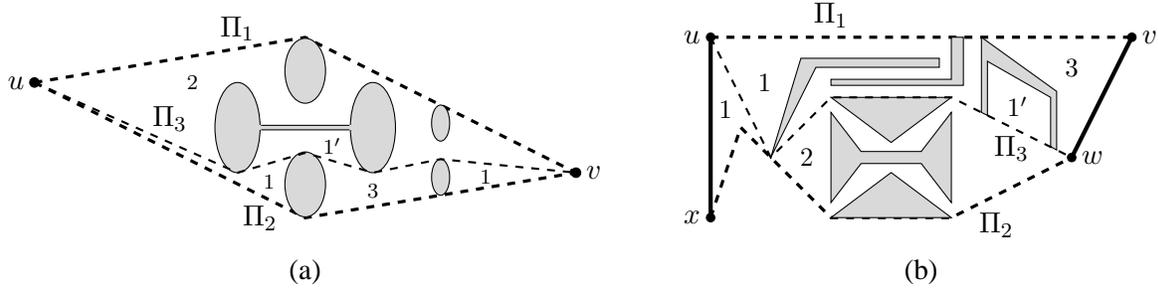
\begin{figure}

\begin{center}

\begin{tabular}{ccc}

\begin{tikzpicture}[scale=1.2]

\draw[dashed, very thick] (1,.5) -- (4,1) -- (7,-.5);
\draw[dashed, very thick] (1,.5) -- (4,-1) -- (7,-.5);

\draw[fill=gray!30] (4, .63) ellipse (.225 and .36);
\draw[fill=gray!30] (4, -.63) ellipse (.225 and .36);

\draw[fill=gray!30] (3.25, 0) ellipse (.25 and .5);
\draw[fill=gray!30] (4.75, 0) ellipse (.25 and .5);

\draw[fill=gray!30] (5.5, -.55 ) ellipse (.1 and .2);
\draw[fill=gray!30] (5.5, .05 ) ellipse (.1 and .2);


\draw[dashed, thick] (1,.5)  -- (3.25, -.5) -- (4,-.28) -- (4.75, -.5) -- (5.5,-.35) -- (7,-.5);

\fill[gray!30] (3.4,.025) -- (4.6,.025) -- (4.6,-.025) -- (3.4,-.025) -- cycle; 
\draw (3.51,.025) -- (4.49,.025);
\draw (3.51,-.025) -- (4.49,-.025);

\node at (2.75,.5) {\scriptsize $2$};
\node at (4.3,-.2) {\scriptsize $1'$};
\node at (4.75,-.7) {\scriptsize $3$};
\node at (3.6,-.6) {\scriptsize $1$};
\node at (6,-.55) {\scriptsize $1$};

\node[left] at (1,.5) {$u$};
\node[right] at (7,-.5) {$v$};

\draw[fill] (1,.5) circle (1.5pt);
\draw[fill] (7,-.5)  circle (1.5pt);

\node at (2.5,.1) {$\Pi_3$};
\node at (3.25,1.1) {$\Pi_1$};
\node at (3.5,-1) {$\Pi_2$};

\node at (4,-1.6) {(a)};

\end{tikzpicture}

& \qquad &

\begin{tikzpicture}[scale=.8]

\draw[ultra thick] (0,0) -- (0,3);
\draw[ultra thick] (6,1) -- (7,3);

\draw[dashed, very thick] (0,3) -- (7,3);
\draw[dashed, very thick] (0,0) -- (.5,1.5) -- (1,1) -- (2,0) -- (4,0) -- (6,1);

\draw[fill=gray!30]  (1,1) --  (1.75, 2.5) -- (3.8, 2.5)  -- (3.8, 2.65) -- (1.5, 2.65)  -- cycle;

\draw[fill=gray!30] (4, 3) -- (4, 2.3) -- (2, 2.3) -- (2, 2.2) -- (4.2, 2.2) -- (4.2, 3) -- cycle;

\draw[fill=gray!30] (2,0) -- (4,0) -- (3, .75) -- cycle;
\draw[fill=gray!30] (2,2) -- (4,2) -- (3, 1.25) -- cycle;

\draw[fill=gray!30] (4.5, 1.75) -- (4.5, 3) -- (5.75, 2.1) -- (5.75, 1.125) -- (5.65, 1.18)-- (5.65, 2) -- (4.6, 2.6) -- (4.6, 1.7) -- cycle;

\draw[fill=gray!30] (2,1.75) -- (2.5,1.1) -- (3.5, 1.1) -- (4, 1.75)  -- (4, .25) -- (3.5, .9) -- (2.5, .9) -- (2, .25) -- cycle;

\draw[dashed, thick] (0,3) -- (1,1) -- (2,2) -- (4,2) -- (6,1);

\node[left] at (0,3) {$u$};
\node[above] at (2,3) {$\Pi_1$};
\node[right] at (7,3) {$v$};

\node[left] at (0,0) {$x$};
\node[below] at (4.75,.25) {$\Pi_2$};
\node[right] at (6,1) {$w$};

\draw[fill] (0,3) circle (2pt);
\draw[fill] (7,3) circle (2pt);
\draw[fill] (0,0) circle (2pt);
\draw[fill] (6,1) circle (2pt);

\node at (5,1.15) {$\Pi_3$};

\node at (.25,1.80) {\small $1$};
\node at (1.6,1.0) {\small $2$};
\node at (0.9,2.25) {\small $1$};
\node at (6,2.5) {\small $3$};
\node at (5.1,1.85) {\small $1'$};

\node at (3.5,-.9) {(b)};

\end{tikzpicture}

\end{tabular}

\caption{Examples of  moves where the new guarded path $\Pi_3$ divides the region into five subregions, each identified by its type. (a) A type 2 move. (b) A type 3 move.}

\label{fig:type2-type3}

\end{center}

\end{figure}

This brings our discussion to a type 3 region, with $p_1$ guarding shortest $(u,v)$-path $\Pi_1$ and $p_2$ guarding shortest $(w,x)$-path $\Pi_2$.
Our primary goal is to trap the evader in a type 1 region, but we might end up transitioning an obstacle instead.   In the latter case, the new evader territory can be of any type, as explained below.

\begin{lemma}
\label{lemma:type3region}
Suppose that the evader is trapped in a type 3 region. Then the third pursuer can guard a path so that either (a) the evader is trapped in a type 1 region, or (b) an obstacle transitions to a lower threat level.
\end{lemma}

\begin{proof}
 Let $\Pi_3$ be a $(u,w)$-shortest path. Pursuer $p_3$ moves to guard this path using Lemma \ref{lemma:guarding}. This traps the evader in a smaller region: without loss of generality, this region is bounded by $\Pi_1, \Delta_1, \Pi_3$. If $\Pi_3$ does not touch any obstacles in this region, then the evader is now in a type 1 region.  
If $\Pi_3$ touches an obstacle $O$, then this obstacle transitions to either safe or removed. 
The evader could be trapped in a region of any type, as shown in  Figure \ref{fig:type2-type3} (b). 
\end{proof}

We can now prove our main theorem: three pursuers can capture the evader in a two-dimensional compact domain with piecewise analytic boundary. 

\begin{proofof}{Theorem \ref{thm:3pursuers}}
The first round traps the evader in a type $1'$ region, or transitions an obstacle state. 
If we are in a region of type 1, $1'$ or 2 then we transition an obstacle state in the current round by Lemma \ref{lemma:type1region} and 
Lemma \ref{lemma:type2region}. When we are in a type 3 region, Lemma \ref{lemma:type3region} ensures that we either trap the evader in a type 1 region, or we transition an obstacle. With each path that we guard, the boundary of the updated evader territory is still piecewise analytic by Lemma \ref{lem: shortest path is nice}. At the end of the round, we update the evader territory and our value for minimum obstacle separation since our new guarded path might be closer to an obstacle than the current value $d_{\min}$. (Note that the maximum boundary curvature $\kappa_{\max}$ never increases since all additions to the boundary are line segments.) 

After at most $2k$ rounds, we have transitioned all $k$ obstacles to either safe or removed. Once all obstacles  have been transitioned, the evader is trapped in a simply connected type 0 region. Lemma \ref{lemma:guardable pair} shows that the evader will then be caught. Each round completes in finite time, so the three pursuers win the game. 
The capture time upper bound of $O( 2k \cdot \diam(\Domain) + \diam(\Domain)^2)$ follows easily. The time required to guard any shortest path is $\diam(\Domain)$ by Lemma \ref{lemma:guarding} and lion's strategy completes in time $\diam(\Domain)^2$ by Theorem \ref{thm:lion}.
\end{proofof}

 \section{Conclusion}

In this paper, we described a winning pursuer strategy for a single pursuer in a $\cat$ space when the capture criterion is $d(p,e)=0$. We then restricted our attention to compact domains in $\euc^2$ with piecewise analytic boundary. We showed that three pursuers are sufficient to catch an evader in such environments.  

There are  plenty of avenues for reseach in topological pursuit-evasion. Pursuit-evasion results on polyhedral surfaces are an active area of current research \cite{klein+suri+polyhedra, noori+isler+terrain,noori+isler+polyhedra}.
For example, Klein and Suri \cite{klein+suri+polyhedra} have proven that $4g+4$ pursuers have a winning strategy on a polyhedral surface of genus $g$. Meanwhile, Schr\"{o}der \cite{schroeder} has proven the that at most$\lfloor 3g/2 \rfloor + 2$ pursuers are needed for a graph of genus $g$ (meaning that such a graph can be drawn on a surface of genus $g$ without edge crossings). It would be natural to consider this question for topological surfaces. 
Likewise, there are a wealth of motion and sensory constraints to consider. Most of these variations of pursuit-evasion have a natural analog in a topological setting.

\section{Acknowledgments}

This work was completed while the first author was a long-term visitor at the Institute for Mathematics and its Applications. We 
are grateful to the IMA for its support and for fostering such a vibrant research community.


\bibliography{citation_pursuit_evasion}

\end{document}